\documentclass[a4paper]{amsart}
\makeatletter
\usepackage{fullpage}
\usepackage{geometry}
\newgeometry{margin=1.7in}
\usepackage{amsmath, amssymb, amsfonts, amstext, verbatim, amsthm, mathrsfs}
\usepackage[mathcal]{eucal}
\usepackage{microtype}
\usepackage{graphicx}
\usepackage[all]{xy}
\usepackage{tikz} 

\usepackage{enumerate}
\usepackage{enumitem}
\usepackage{xspace}

\usepackage[modulo]{lineno}
\usepackage{aliascnt} 
\usepackage[colorlinks=true,linkcolor=blue,citecolor=blue,urlcolor=blue,citebordercolor={0 0 1},urlbordercolor={0 0 1},linkbordercolor={0 0 1}]{hyperref} 
\usepackage[nameinlink]{cleveref}

\def\makeCal#1{%
\expandafter\newcommand\csname c#1\endcsname{\mathcal{#1}}}
\def\makeBB#1{%
\expandafter\newcommand\csname b#1\endcsname{\mathbb{#1}}}
\def\makeFrak#1{%
\expandafter\newcommand\csname f#1\endcsname{\mathfrak{#1}}}	

\count@=0
\loop
\advance\count@ 1
\edef\y{\@Alph\count@}%
\expandafter\makeCal\y
\expandafter\makeBB\y
\expandafter\makeFrak\y
\ifnum\count@<26
\repeat
\numberwithin{equation}{section}

\theoremstyle{plain}

\newtheorem{theorem}{Theorem}[section]

\newtheorem{lem}[theorem]{Lemma}

\newtheorem{prop}[theorem]{Proposition}

\theoremstyle{definition}

\newtheorem{defn}[theorem]{Definition}

\theoremstyle{remark}

\newtheorem{remark}[theorem]{Remark}

\DeclareMathOperator{\Spec}{Spec}
\DeclareMathOperator{\ST}{ST}
\DeclareMathOperator{\GL}{GL}

\DeclareMathOperator{\Aut}{Aut}
\DeclareMathOperator{\Ext}{Ext}

\DeclareMathOperator{\Norm}{Norm}

\newcommand{\colim}{\varprojlim}

\newcommand{\ilim}{\varprojlim}

\newcommand{\iso}{\stackrel{\sim}{\to}}

\newcommand{\co}{\colon}

\newcommand{\oh}{\cO}
\newcommand{\fm}{\mathfrak{m}}

\DeclareMathOperator{\Isom}{Isom}

\DeclareMathOperator{\Map}{Map}

\newdir{|>}{!/4.5pt/@{|}*:(1,-.2)@^{>}*:(1,+.2)@_{>}}

\usepackage[normalem]{ulem}
\usepackage{soul}
\usepackage{pbox}

\begin{document}
\title{Cartan--Iwahori--Matsumoto decompositions for reductive groups}
\author[J.~Alper]{Jarod Alper}
\address{Department of Mathematics \\ University of Washington \\ Box 354350 \\ Seattle, WA 98195-4350, USA}
\email{jarod@uw.edu}
\author[D.~Halpern-Leistner]{Daniel Halpern-Leistner}
\address{Malott Hall \\ Mathematics Dept \\
Cornell University  \\
Ithaca, NY 14853, USA}
\email{daniel.hl@cornell.edu}
\author[J.~Heinloth]{Jochen Heinloth}
\address{Universit\"at Duisburg-Essen \\
Fakult\"at f\"ur Mathematik \\
45117 Essen, Germany}
\email{jochen.heinloth@uni-due.de}
\begin{abstract}
We provide a short and self-contained argument for the existence of Cartan--Iwahori--Matsumoto decompositions for reductive groups.
\end{abstract}
\maketitle

\section{Introduction}

The following classical theorem on the existence of Cartan--Iwahori--Matsumoto decompositions for reductive groups is the key algebraic input in the proof of the Hilbert--Mumford criterion in geometric invariant theory (see \cite[p.52]{git}):

\begin{theorem} \label{T:theorem1} 
{Let $R$ be a complete DVR with fraction field $K$ and algebraically closed residue field $k$. Let $G$ be a reductive group scheme over $R$.   For any element $g \in G(K)$, there exists elements $h_1, h_2 \in G(R)$ and a one-parameter subgroup $\lambda\co \bG_{m,R} \to G$ such that $g= h_1\lambda|_{K} h_2$.}
\end{theorem}

\begin{remark}  In the above theorem, $\lambda|_K \in G(K)$ represents the composition $\Spec(K) \to \bG_{m,R} \xrightarrow{\lambda} G$, where the first map is defined by mapping a coordinate of $\bG_{m,R}$ to a uniformizing parameter in $R \subset K$.
\end{remark}

An easy linear algebra argument establishes the theorem in the case that $G = \GL_n$ (c.f. \cite[Lem.\ 7.7]{mukai-moduli}).  For 
 semi-simple algebraic groups of adjoint type over an algebraically closed field, the theorem was established in \cite[Cor. 2.17]{iwahori-matsumoto}; Mumford observed in \cite[p.52]{git} that this case is sufficient to imply the result for reductive groups in characteristic 0.   A slightly weaker version of \Cref{T:theorem1}, where finite extensions of the DVR are allowed, was established for reductive groups in positive characteristic in \cite[Thm. 2.1]{seshadri-quotient} (see also \cite[Appendix 2.A]{git2}).   See also \cite[\S 4]{bruhat-tits1} for analogous decomposition results for ramified groups. 

We provide a short and self-contained proof of \Cref{T:theorem1} in \Cref{S:elementary-proof} using a reduction argument to the case of $\GL_n$.  This argument was inspired by the stack-theoretic condition of S-completeness introduced in \cite{ahlh}.  In fact, our methods are sufficiently general to remove the condition in \Cref{T:theorem1} that the residue field be algebraically closed as well as establish a converse statement providing a characterization of reductivity.  Our main theorem is:

\begin{theorem} \label{T:theorem2}  
If $G \to S$ is a smooth affine group scheme over a noetherian scheme $S$, the following are equivalent:
\begin{enumerate}
	\item $G^0 \to S$ is reductive and $G/G^0 \to S$ is finite;
	\item $B_SG \to S$ is S-complete (see \Cref{D:S-complete}); and
	\item $G \to S$ has Cartan decompositions with respect to any complete DVR over $S$ (see \Cref{D:iwahori}).
\end{enumerate}
\end{theorem}

The above characterization of reductivity in terms of the existence of Cartan decompositions was recently also discovered by Chenyang Xu and Jun Yu.

While it is possible to both formulate and prove \Cref{T:theorem2} avoiding the language of algebraic stacks using arguments similar to those in \Cref{S:elementary-proof}, we find that algebraic stacks provide a natural and powerful language which allows for a conceptual proof of \Cref{T:theorem2}.  We prove this theorem in \Cref{S:proof}. In \Cref{S:Non-Complete-DVR} we use this result to deduce a decomposition over non-complete DVR's, in which case an additional hypothesis on $G$ is required.

\subsection*{Acknowledgements} We thank Xiaowei Wang and Chenyang Xu for helpful conversations.  The first author was partially supported by NSF grant DMS-1801976. The second author was partially supported by NSF grant DMS-1762669. The third author was partially supported by Sonderforschungsbereich/Transregio 45 of the DFG.

\section{Short Proof of \Cref{T:theorem1}} \label{S:elementary-proof}

\begin{proof}[Proof of \Cref{T:theorem1}]
Let $\pi \in R$ be a uniformizing parameter and consider the affine scheme $X=\Spec(R[s,t]/(st-\pi))$ endowed with the action of $\bG_m$ where $s$ and $t$ have weights $1$ and $-1$, respectively.  Let $0 \in X$ be the point corresponding to the maximal ideal $\fm = (s,t)$.  

\bigskip
\noindent {\it Claim 1:} Any  $g \in G(K)$ determines a $G$-torsor $\cP_g \to X \smallsetminus 0$ with a left $\bG_m$-action commuting with the right $G$-action and compatible with the $\bG_m$-action on $X \smallsetminus 0$.

To see this, observe that $X \smallsetminus 0$ is the union of the $\bG_m$-invariant open subschemes $X_s = \Spec(R[s]_s) \cong \Spec(R) \times \bG_m$ and $X_t = \Spec(R[t]_t) \cong \Spec(R) \times \bG_m$ along $X_{st} = \Spec(K[s]_s) \cong \Spec(K) \times \bG_m$.  If we consider the trivial $G$-torsors $X_s \times G \to X_s$ and $X_t \times G \to X_t$ with the $\bG_m$-action which is trivial on each copy of $G$, then the element $g \in G(K)$ yields a $\bG_m$-equivariant isomorphism of their restrictions to $X_{st}$ and therefore a $G$-torsor $\cP_g \to X \smallsetminus 0$ with a $\bG_m$-action.

\bigskip
\noindent {\it Claim 2:} There is a $\bG_m$-equivariant extension of $\cP_g$ to a $G$-torsor $\widetilde{\cP}_g \to X$ with a left $\bG_m$-action commuting with the right $G$-action and compatible with the $\bG_m$-action on the base.

Choose an embedding $G \subset \GL_n$.  
Giving a $G$-torsor on a scheme $T$ is equivalent to giving a $\GL_n$-torsor $E \to T$ and a section of $E/G \to T$.  Under this correspondence, we are given a $\GL_n$-torsor $E \to X \smallsetminus 0$ and a section $X \smallsetminus 0 \to E/G$.  
As $X$ is regular and $0 \in X$ is codimension 2, there is a unique extension of $E \to X \smallsetminus 0$ to a $\GL_n$-torsor $\widetilde{E} \to X$; indeed, we can translate this question into the analogous extension property for vector bundles and if $\cV$ is a vector bundle on $X \smallsetminus 0$, then $(X \smallsetminus 0 \to X)_* \cV$ is the unique extension.  As $G$ is reductive, $\widetilde{E}/G$ is affine.  Since $\Gamma(X, \oh_X) = \Gamma(X \smallsetminus 0, \oh_{X})$, the section $X \smallsetminus 0 \to \widetilde{E}/G$ extends uniquely to a section $X \to  \widetilde{E}/G$.

\bigskip
\noindent {\it Claim 3:} There is a one-parameter subgroup $\lambda \co \bG_{m,R} \to G$ and a $\bG_m$-equivariant isomorphism of $G$-torsors between $\widetilde{\cP}_g$ and the trivial $G$-torsor $\cP_{\lambda} := \Spec(R[s,t]/(st-\pi) \times G \to \Spec(R[s,t]/(st-\pi)$ with a left $\bG_m$-action where $\bG_m$ acts on the copy of $G$ via multiplication by $\lambda$.

The fiber of $\widetilde{\cP}_g \to X$ over $0 \in X$ is the trivial $G_k$-torsor $G_k \to \Spec(k)$ with a left $\bG_m$-action commuting with the right $G_k$-action. As we assumed $k$ to be algebraically closed the restriction $\widetilde{\cP}_g|_0$ of $\widetilde{\cP}_g$ to the origin is trivial and thus this data determines a one-parameter subgroup $\lambda_0 \co \bG_{m,k} \to \Aut_{G_k}(\widetilde{\cP}_g|_0) \cong G_k$ unique up to conjugation.

Since $R$ is complete, we may lift $\lambda_0$ to a one-parameter subgroup $\lambda \co  \bG_{m,R} \to G$ by  \cite[Exp.~IX, Cor. 7.3]{sga3ii}.
This equips the trivial $G$-torsor over $X$ with a $\bG_m$-action and we denote the corresponding equivariant torsor by $\cP_\lambda$.
There is a $\bG_m$-equivariant isomorphism $\alpha^{[0]}$ between the fibers of the $G$-torsors $\cP_{\lambda}$  and $\widetilde{\cP}_g$ over $0 \in X$.  Let $X^{[n]} \subset X$ be the $n$th nilpotent thickening of $0 \in X$, i.e., the subscheme defined by $\fm^{n+1}$.  As  $\Isom(\cP_\lambda, \widetilde{\cP}_g)$ is smooth over $X$ and $\bG_m$ is linearly reductive, the isomorphism $\alpha^{[0]} \co  \cP_{\lambda}|_{X^{[0]}} \iso \widetilde{\cP}_g|_{X^{[0]}}$ lifts to a compatible family of $\bG_m$-equivariant isomorphisms $\alpha^{[n]} \co \cP_{\lambda}|_{X^{[n]}} \iso \widetilde{\cP}_g|_{X^{[n]}} $.

Finally, we claim that the isomorphisms $\alpha^{[n]}$ extend to a $\bG_m$-equivariant isomorphism $\alpha \co \cP_{\lambda} \to \widetilde{\cP}_g$ of $G$-torsors. Let $\cI = \Isom(\cP_\lambda, \widetilde{\cP}_g)$, which we identify with $\widetilde{\cP}_g$ with its $\bG_m$-action modified by right multiplication by $\lambda^{-1}$. The existence of the extension $\alpha$ is equivalent to giving a $\bG_m$-equivariant section $X \to \cI$ extending the sections $X^{[n]} \to \cI$ induced by $\alpha^{[n]}$. We translate this claim into commutative algebra by writing $A = \Gamma(X, \oh_X)$ and $A^{[n]} = \Gamma(X^{[n]}, \oh_{X^{[n]}}) =  A / \fm^{n+1}$.  The $\bG_m$-action induces $\bZ$-gradings $A = \oplus_d A_d$ and $A^{[n]} = \oplus_d A^{[n]}_d$ where
 $$A_{d} =\left\{
 \begin{array}{rl}
R \langle s^d \rangle & \text{if } d \ge 0\\
 R \langle t^d \rangle & \text{if } d < 0
 \end{array} \right.
 \quad \text{and} \quad
A^{[n]}_d =\left\{
 \begin{array}{rl}
R/(\pi^{\lfloor \frac{n-d}{2} \rfloor + 1}) \langle s^d \rangle & \text{if } n \ge d \ge 0\\
R/(\pi^{\lfloor \frac{n+d}{2} \rfloor + 1}) \langle t^d \rangle & \text{if } n \ge -d > 0 \\
0 & \text{otherwise.} 
 \end{array} \right.
  $$
  As $R$ is complete, we have that $A_d \cong \ilim_n A^{[n]}_d$ for all $d$. 
Similarly, the $\bG_m$-action on $\cI$ induces a $\bZ$-grading $\Gamma(\cI, \oh_{\cI}) = \oplus_d \Gamma(\cI, \oh_{\cI})_d$.  The existence of $\alpha$ is then equivalent to giving a graded homomorphism $\Gamma(\cI, \oh_{\cI}) \to A$ filling in the diagram
$$\xymatrix{
\Gamma(\cI, \oh_{\cI}) \ar@{-->}[r] \ar[rd]	& A \ar[d] \\
							& A^{[n]}
}$$
for each $n \ge 0$.  For each $d$, the compatible maps $\Gamma(\cI, \oh_{\cI})_d \to A^{[n]}_d$ induce a map $\Gamma(\cI, \oh_{\cI})_d \to \ilim_n A^{[n]}_d \cong A_d$, which verifies the existence of $\alpha$.

\bigskip
\noindent {\it Conclusion:} The existence of elements $h_1, h_2 \in G(R)$ such that $g = h_1 \lambda|_K h_2$ now follows from the following two elementary observations: 
\begin{enumerate}
\item[(A)] If $g, g' \in G(K)$ are elements, the $G$-torsors $\cP_g$ and $\cP_{g'}$ on $X \smallsetminus 0$ are $\bG_m$-equivariantly isomorphic if and only if and only if there are elements $h,h' \in G(R)$ such that $hg = g' h'$.  
\item[(B)] If $\lambda \co \bG_{m,R} \to G$ is a one-parameter subgroup, then there is a $\bG_m$-equivariant isomorphism of $G$-torsors on $X \smallsetminus 0$ between $\cP_{\lambda}|_{X \smallsetminus 0}$ and $\cP_{g'}$, where $g' = \lambda|_K$.  
\end{enumerate}
\end{proof}

\section{Proof of \Cref{T:theorem2}} \label{S:proof}

\subsection{S-complete morphisms}

In order to prove \Cref{T:theorem2}, we recall the definition and some basic properties of S-complete morphisms from \cite{ahlh}.  First, if $R$ is a DVR with uniformizing parameter $\pi$, we define the algebraic stack
$$\overline{\ST}_R := [\Spec \big(R[s,t] / (st-\pi)\big)  / \bG_m],$$
where $s$ and $t$ have weight $1$ and $-1$.

\begin{defn} \label{D:S-complete}
We say that a morphism $f \co \cX \to \cY$ of locally noetherian algebraic stacks is \emph{S-complete} if for any DVR $R$ and any commutative diagram
\begin{equation} \label[diagram]{E:S-complete}
\begin{split}
\xymatrix{
\overline{\ST}_R  \smallsetminus 0 \ar[r] \ar[d]					& \cX \ar[d]^f \\
\overline{\ST}_R \ar[r] \ar@{-->}[ur]		& \cY
}
\end{split} \end{equation}
of solid arrows, there exists a unique dotted arrow filling in the diagram.
\end{defn}

The property of being S-complete is stable under base change. The following elementary properties of S-completeness were established in \cite[\S 3.5]{ahlh}.

\begin{prop} \label{P:S-complete} \quad
\begin{enumerate}
\item \label{P:S-complete1} An affine morphism of locally noetherian algebraic stacks is S-complete.
\item \label{P:S-complete2}  $B_{\bZ} \GL_N \to \Spec(\bZ)$ is S-complete. 
\end{enumerate}
Let $f \co \cX \to \cY$ be a morphism of noetherian algebraic stacks. 
\begin{enumerate} \setcounter{enumi}{2}
\item  \label{P:S-complete3}
If $\cX$ and $\cY$ both have quasi-finite and separated inertia, then $f$ is S-complete if and only if $f$ is separated.
\item \label{P:S-complete5} If $\cY' \to \cY$ is an \'etale, representable and surjective morphism, then $\cX \to \cY$ is S-complete if and only if $\cX \times_{\cY} \cY' \to \cY'$ is S-complete.
\item \label{P:S-complete6} If $\cX' \to \cX$ is a finite, \'etale and surjective morphism, then $\cX \to \cY$ is S-complete if and only if $\cX' \to \cY$ is S-complete.
\item \label{P:S-complete4} If $f$ has affine diagonal, then a lifting in \eqref{E:S-complete} is automatically unique, and it suffices to verify the existence of a lifting after passing to an arbitrary extension $R \subset R'$ of DVRs.
\end{enumerate}
\end{prop}

\subsection{A coherent completeness result}

To prove \Cref{T:theorem2}, we will replace Claim 3 in the proof of \Cref{T:theorem1} with the following consequence of coherent completeness (\cite[Def.~3.3]{ahr}) and Tannaka duality:

\begin{lem} \label{L:full}
Let $R$ be a complete DVR with residue field $k$.  Let $G \to \Spec(R)$ be a smooth affine group scheme. The restriction functor of groupoids
$$\Map(\overline{\ST}_R, B_R G) \to \Map(B_k \bG_m, B_R G),$$
induced by the inclusion $B_k \bG_m \subset \overline{\ST}_R$, is full.
\end{lem}

\begin{proof}
The stack $\overline{\ST}_R$ is coherently complete along the residual gerbe $B_k \bG_m$ of the closed point $0 \in \overline{\ST}_R$ by \cite[Thm.~1.3]{ahr}.  Let $\cI$ be the ideal sheaf defining $B_k \bG_m \subset \overline{\ST}_R$ and let $\cZ^{[n]} \subset \overline{\ST}_R$ be the closed substack defined by $\cI^{n+1}$.  Consider maps $f,f' \co \overline{\ST}_R \to BG$ and a 2-isomorphism $\alpha^{[0]} \co f|_{\cZ^{[0]}} \iso f'|_{\cZ^{[0]}}$.  By \cite{olsson-defn}, the obstruction to extending a 2-isomorphism $\alpha^{[n]} \co f|_{\cZ^{[n]}}\iso f'|_{\cZ^{[n]}}$ to $\alpha^{[n+1]} \co f|_{\cZ^{[n+1]}}\iso f'|_{\cZ^{[n+1]}}$ lies in the group $\Ext^0(Lf_0^* L_{BG/R}, \cI^n / \cI^{n+1})$ where $f_0 \co B_k \bG_{m} \to \overline{\ST}_R \to BG$. Since $B_k \bG_m$ has no higher coherent cohomology and $L_{BG/R}$ is supported in degree 1 (as $G$ is smooth), the obstruction vanishes and we obtain compatible 2-isomorphisms between $f|_{\cZ^{[n]}}$ and $f'|_{\cZ^{[n]}}$ for every $n$.  A special case of Tannaka duality (e.g. \cite[Cor.~3.6]{ahr}) asserts that the restriction map
$$\Map(\overline{\ST}_R, BG) \to \colim \Map(\cZ^{[n]}, BG)$$
is an equivalence of categories.  It follows that there is a 2-isomorphism $\alpha \co f \iso f'$ extending $\alpha^{[0]}$.
\end{proof}

\subsection{Cartan decompositions}

Let $G$ be a group scheme over a DVR $R$ with fraction field $K$ and uniformizing parameter $\pi$.  If $\lambda \co \bG_{m,R} \to G$ is a homomorphism of group schemes, then we write $\lambda|_K \in G(K)$ for the composition of $\Spec(K) \to \bG_{m,R} = \Spec(R[t]_t)$, defined by $t \mapsto \pi$, with $\lambda$.

\begin{defn} \label{D:iwahori} Let $G \to S$ be a smooth affine group scheme over a noetherian scheme $S$.
Let $R$ be a DVR over $S$ with fraction field $K$. 
We say that {\it $G$ has Cartan decompositions\footnote{It would be more accurate to call these `Cartan--Iwahori--Matsumoto decompositions.'  We chose the above notation for the sake of brevity.} with respect to $R$} if any element $g \in G(K)$ can be written as $g = h_1 \lambda|_K h_2$ where $h_1, h_2 \in G(R)$ and $\lambda \co \bG_{m,R} \to G_R$ is a one-parameter subgroup.
\end{defn}

\begin{remark} \label{R:iwahori}
If $S$ is the spectrum of a  DVR $R$ and $T \subset G$ is a maximal split torus over $R$, then $G$ has Cartan decompositions with respect to $R$ if and only if 
$$G(K) = G(R) T(K) G(R).$$
For the `$\Rightarrow$' direction, we may find an element $h \in G(R)$ such that the image of $h \lambda|_K h^{-1}$ is contained in $T$ by \cite[Exp.~XXVI, Prop. 6.16]{sga3iii}. 
Then 
$$g = h_1 \lambda|_K h_2 = \underbrace{(h_1 h^{-1})}_{\in G(R)}\underbrace{(h \lambda|_K h^{-1})}_{\in T(K)}\underbrace{(h h_2)}_{\in G(R)}.$$

  Conversely, suppose $g \in G(K)$ can be written as $g = h_1 t h_2$ for $h_1, h_2 \in G(R)$ and $t \in T(K)$.  If we write $T \cong \bG_m^r$ and $\pi \in R$ as the uniformizing parameter, then $t =  (u_1 \pi^{d_1}, \ldots, u_r \pi^{d_r})$ for units $u_i \in R^{\times}$ and integers $d_i \in \bZ$.  After replacing $h_1$ with $h_1 (u_1, \ldots, u_r)$, we can write $g = h_1 \lambda|_K h_2$ where $\lambda \co \bG_{m,R} \to T \subset G$ is the one-parameter subgroup given by $t \mapsto (t^{d_1}, \ldots, t^{d_r})$.
\end{remark}

As $\overline{\ST}_R \smallsetminus 0 = \Spec(R) \bigcup_{\Spec(K)} \Spec(R)$, an element $g \in G(K)$ determines a morphism 
$$\rho_g \co \overline{\ST}_R \smallsetminus 0 \to B_SG$$ 
by gluing two trivial $G$-torsors over $\Spec(R)$ via the isomorphism induced by $g$ of their restrictions to $\Spec(K)$.

\begin{lem} \label{L:Iwahori-vs-S-complete}
 Let $G \to S$ be a smooth affine group scheme over a noetherian scheme $S$.
Let $R$ be a complete DVR over $S$ with fraction field $K$.  For any element $g \in G(K)$, the following are equivalent:
\begin{enumerate}
\item $g$ can be written as $g = h_1 \lambda|_K h_2$ where $h_1, h_2 \in G(R)$ and $\lambda \co \bG_{m,R} \to G_R$ is a one-parameter subgroup; and
\item there exists a dotted arrow filling in the commutative diagram
$$\xymatrix{
\overline{\ST}_R \smallsetminus 0 \ar[r]^{\rho_g} \ar[d]	&	B_SG \\
\overline{\ST}_R .  \ar@{-->}[ur]
}$$
\end{enumerate}
\end{lem}

\begin{proof}
We begin with making two observations:
\begin{enumerate}
\item[(A)] If $g, g' \in G(K)$ are elements, the morphisms $\rho_g, \rho_{g'} \co \overline{\ST}_R \smallsetminus 0 \to B_SG$ are 2-isomorphic if and only if there are elements $h,h' \in G(R)$ such that $hg = g' h'$.  
\item[(B)] If $\lambda \co \bG_{m,R} \to G_R$ is a one-parameter subgroup and $\widetilde{\lambda}$ denotes the composition $\widetilde{\lambda} \co \overline{\ST}_R \to B_R \bG_{m} \to B_S G$ where the latter map is induced by $\lambda$, then $\widetilde{\lambda}_{\overline{\ST}_R \smallsetminus 0}$ and $\rho_{g'}$, where $g' = \lambda|_K$, are 2-isomorphic.
\end{enumerate}
To see (1) $\implies$ (2), Observations (A) and (B) imply that $\widetilde{\lambda} \co \overline{\ST}_R \to B_R \bG_{m} \to B_SG$ is an extension of $\rho_g$.   

For the converse (2) $\implies$ (1), we may restrict an extension $\widetilde{\rho} \co \overline{\ST}_R \to B_SG$ to the residual gerbe of the closed point $0 \in \overline{\ST}_R$ to obtain a morphism $B_k \bG_m \to B_SG$, where $k$ is the residue field of $R$.  This yields a homomorphism $\bG_m \to G'$ where $G'$ is the inner form of $G$ defined by $\Aut_{B_S G}(k)$.  We claim that this inner form is trivial.  Indeed, we may restrict the $G$-torsor corresponding to $\widetilde{\rho}$ to $\{s=0\}$ to obtain a $G$-torsor $\cP$ on $[\bA^1_k / \bG_m]$ which is trivial over $1$.  Let $\cP'$ denote the $G$-torsor $[\bA^1_k / \bG_m] \to B_k \bG_m \to B_S G$. Since $\Isom(\cP, \cP') \to [\bA^1_k / \bG_m]$ is smooth and we have a section over $B \bG_m$, we may use Tannaka duality (\cite[Cor.~3.6]{ahr}) to lift this section to $[\bA^1_k / \bG_m]$.  This yields an isomorphism $\cP' \iso \cP$ which restricts over $1$ to an isomorphism  $\cP'|_1 \cong \cP|_1 \cong G$ so $G'$ is trivial.

Let $\lambda_0 \co \bG_{m,k} \to G_k$ be the corresponding one-parameter subgroup.  By \cite[Exp.~IX, Cor. 7.3]{sga3ii}, $\lambda_0$ lifts to a homomorphism $\lambda \co  \bG_{m,R} \to G_R$. 
Consider the composition $\widetilde{\lambda} \co \overline{\ST}_R \to B_R \bG_m \to B_SG$, where the latter map is induced by $\lambda$.  The restrictions $\widetilde{\rho}|_{B_k \bG_m}$ and $\widetilde{\lambda}|_{B_k \bG_m}$ are 2-isomorphic.  It follows from \Cref{L:full} that $\widetilde{\rho}$ and $\widetilde{\lambda}$ are 2-isomorphic and, in particular, $\rho_g$ and $\widetilde{\lambda}|_{\overline{\ST}_R \smallsetminus 0}$ are 2-isomorphic.   Observations (A) and (B) now imply that we may write $g$ as $g = h_1 \lambda|_K h_2$ where $h_1, h_2 \in G(R)$.
\end{proof}

\begin{proof}[Proof of \Cref{T:theorem2}]
\noindent (1)$ \implies$ (2):  As $B_S G^0 \to B_S G$ is a finite \'etale covering and S-completeness is local on the source under finite  \'etale coverings (\Cref{P:S-complete}\eqref{P:S-complete6}), we may assume that $G \to S$ is reductive.  
Since $S$-completeness is \'etale local (\Cref{P:S-complete}\eqref{P:S-complete5}) and reductive group schemes are \'etale-locally split, we may assume that $G$ embeds as a closed subgroup scheme of $\GL_{n,S}$.  As $G \to S$ is reductive, the quotient $\GL_{n,S}/G$ is affine or, in other words, $B_SG \to B_S\GL_n$ is affine.  As $B_S\GL_n$ is $S$-complete (\Cref{P:S-complete}\eqref{P:S-complete2}), $B_SG$ is also $S$-complete  (\Cref{P:S-complete}\eqref{P:S-complete1}).

\medskip

\noindent (2) $\implies$ (1):  Suppose that there is a geometric point $s \co \Spec(k) \to S$ such that $G_s^0$ is not reductive.    
Choose a normal subgroup $\bG_a \subset R_u(G_{s}^0)$ of the unipotent radical.  
As both $G_s^0/R_u(G_s^0)$ and $R_u(G_{s}^0)/\bG_a$ are affine, the morphism $B_k \bG_a \to B_kG$ is affine but this implies by \Cref{P:S-complete}\eqref{P:S-complete1} that $B_k \bG_a$ is S-complete which is a contradiction.  Since the fibers of $G^0 \to S$ are reductive,  \cite[Exp.~XVI, Thm. 5.2]{sga3ii} implies that $G^0 \to S$ is necessarily affine and thus a reductive group scheme.

To see that $G/G^0 \to S$ is finite, we first observe that $B_S (G/G^0) \to S$ is S-complete.  Indeed, by \Cref{P:S-complete}\eqref{P:S-complete4}, we only need to check that any diagram \eqref{E:S-complete}, where $R$ a complete DVR with algebraically closed residue field, can be filled in after an extension of DVRs.  A morphism $\overline{\ST}_R \smallsetminus 0 \to B_S(G/G^0)$ corresponds to an element of $g \in (G/G^0)(K)$ which after a finite extension of $K$ lifts to an element $\widetilde{g} \in G(K)$.  As $B_S G \to S$ is S-complete, we can extend the morphism $\overline{\ST}_R \smallsetminus 0 \to B_S G$ induced by $\widetilde{g}$ to a morphism $\overline{\ST}_R \to B_SG$ and the composition $\overline{\ST}_R \to B_SG \to B_S(G/G^0)$ yields the desired extension.  Finally, we appeal to  \Cref{P:S-complete}\eqref{P:S-complete3} to conclude that $B_S(G/G^0) \to S$ is separated which implies that $G/G^0 \to S$ is finite.
\medskip

\noindent (2) $\implies$ (3): This follows from \Cref{L:Iwahori-vs-S-complete}.

\medskip

\noindent (3) $\implies$ (2):  By \Cref{P:S-complete}\eqref{P:S-complete4}, we only need to show that for complete DVRs $R$ with algebraically closed residue field, any map $\rho \co \overline{\ST}_R \smallsetminus 0 \to B_SG$ extends to a map $\overline{\ST}_R \to BG$.  As $G$ is smooth, the restrictions $\rho|_{s \neq 0}, \rho|_{t \neq 0} \co \Spec(R) \to B_SG$ correspond to trivial $G$-torsors. Thus $\rho$ is 2-isomorphic to $\rho_g$ for an element $g \in G(K)$.  \Cref{L:Iwahori-vs-S-complete} now implies the existence of an extension. 
\end{proof}

\section{A Result for non-complete DVR's}\label{S:Non-Complete-DVR}

\Cref{T:theorem2} generalizes \Cref{T:theorem1} in that the assumption on the residue field of the complete DVR is dropped. Let us mention that although for a general non-complete DVR the Cartan decomposition may fail even for tori, it turns out that the only obstruction comes from the proof of \Cref{L:Iwahori-vs-S-complete} in which deformation theory was used to spread out a one-parameter subgroup over the residue field $\lambda_0 \co \bG_{m,k}\to G_k$ to $\lambda\co \bG_{m,R} \to G_R$.

\begin{theorem}\label{T:non-complete-R}
	Let $R$ be a DVR with fraction field $K$ and residue field $k$ and {$G$} a {reductive group scheme} over $R$. Suppose that a maximal split torus $S_k$ of the special fiber $G_k$ admits a lift to a split torus $S_R \subset {G}$. Then {$G$ has Cartan decompositions with respect to $R$.}
	\end{theorem}
\begin{remark}
The hypothesis of the above result is clearly satisfied if either $G$ is a split group scheme or if the special fiber $G_k$ is anisotropic. 
It also holds if $R$ is a $k$-algebra and $G$ is a constant group scheme over $R$.
\end{remark}	
\begin{remark}\label{R:RelativeWeylGroup}
As all maximal split tori are conjugate over fields the assumption implies that any one-parameter subgroup of the special fiber $G_k$ is conjugate to the reduction of a one-parameter subgroup $\bG_{m,R}\to S_R\subset G$. 

Also as $R$ is a local ring we know that under this assumption the relative Weyl group $W_{G}(S_R)$ admits representatives $w\in \Norm_G(S_R)(R)\subset G(R)$ by \cite[Exp.~XXVI, \S 7.1]{sga3iii}.
\end{remark}
\begin{proof}
Let us denote by $\widehat{R}$ the completion of $R$ and by $\widehat{K}$ its fraction field. By \Cref{T:theorem2} we can write 
$$g = \widehat{h}_1 \lambda|_K \widehat{h}_2$$ with $\widehat{h}_i \in G(\widehat{R})$ and by our assumption we can choose $\lambda$ to be a one-parameter subgroup $\lambda \co \bG_{m,R} \to S_R \subset G$.

Our aim is to show that we can replace $\widehat{h}_i$ by elements of $G(R)$.  To achieve this let us denote by $P^+_\lambda \subset G$ the parabolic subgroup defined by $\lambda$ and by $U_\lambda^-$ the unipotent radical of the opposite parabolic.	 

Then by the Bruhat decomposition of $G_k$ and \Cref{R:RelativeWeylGroup} for some $w\in G(R)$ the element $\widehat{h}_2$ is an element of the open subset $P^+_\lambda w w^{-1}U^{-}_\lambda w \subset G$ and we can thus write $\widehat{h}_2 =\widehat{p} w \widehat{u}^w$ with $\widehat{p}\in P^+_\lambda(\widehat{R}), \widehat{u} \in U^-_\lambda(\widehat{R})$. As conjugation with $\lambda$ preserves $P_\lambda^+(R)$ we find
$$ g = \widehat{h}_1 \lambda|_K \widehat{h}_2 = \widehat{h}_1 \lambda|_K \widehat{p} \cdot w \cdot\widehat{u}^w =  \widehat{h}_1\cdot\widehat{p}^\prime\cdot \lambda|_K {\cdot w \cdot} \widehat{u}^w$$
for some $\widehat{p}' \in P_{\lambda}^+(\widehat{R})$.
Now $U^-_\lambda$ being the unipotent radical of a parabolic subgroup it is isomorphic as a scheme to an affine space over $R$ (\cite[Exp.~XXVI Cor.~2.5]{sga3iii}) and $\lambda$ acts linearly on this space.
Thus for any integer $n$ we can find $u\in U^-_{\lambda}(R)$ such that $\widehat{u}^w= \widehat{v}^w \cdot u^{w }$ and $\widehat{v}^w\equiv 1 \mod \pi^n$. For sufficiently large $n$ this implies that ${ \widehat{v}^\prime =}\lambda|_K w \widehat{v}^w w^{-1}  \lambda|_K^{-1} \in G(\widehat{R})\subset G(\widehat{K})$  so that we find $$ g= {\widehat{h}_1 \widehat{p}^\prime \widehat{v}^\prime }\lambda|_K wu^w = \widehat{h}_1^{\prime\prime}\lambda|_K wu^w $$
for some $\widehat{h}_1^{\prime\prime} \in G(\widehat{R})$.
This implies that $\widehat{h}_1^{\prime\prime} \in G(K)\cap  G(\widehat{R}) = G(R)$. Thus we found a decomposition over $R$. 
\end{proof}

\bibliographystyle{bibstyle}
\bibliography{refs-iwahori}

\end{document}